\documentclass[12pt,a4paper]{report}
\usepackage[utf8]{inputenc}
\usepackage[french]{babel}
\usepackage{graphicx}
\usepackage{layout}
\usepackage{color}
\usepackage[T1]{fontenc}
\usepackage{stix2}
\usepackage{enumitem}
\usepackage{hyperref}
\usepackage{amsfonts}
\usepackage[all]{xy}
\usepackage{stmaryrd}
\usepackage{amsmath,amscd}
\hypersetup{
    colorlinks=true,
    linkcolor=blue,
    filecolor=magenta,
    urlcolor=cyan,
    citecolor=blue,
}

\usepackage[left=2.4cm,right=2.4cm,top=3cm,bottom=4cm]{geometry}

\usepackage{hyperref}
\usepackage{fancyhdr}
\usepackage{amsmath,amsfonts,amsthm}
\usepackage{setspace}
\usepackage[utf8]{inputenc}
\usepackage[french]{babel}
\usepackage{graphicx}
\usepackage{layout}
\usepackage{color}
\usepackage[T1]{fontenc}
\usepackage{stix2}
\usepackage{enumitem}
\usepackage{hyperref}
\usepackage{amsfonts}
\usepackage[all]{xy}
\usepackage{stmaryrd}
\usepackage{amsmath,amscd}
\hypersetup{
    colorlinks=true,
    linkcolor=blue,
    filecolor=magenta,
    urlcolor=cyan,
    citecolor=blue,
}

\usepackage[left=2.4cm,right=2.4cm,top=3cm,bottom=4cm]{geometry}

\usepackage{hyperref}

\renewcommand\thechapter{\arabic{chapter}}

\newtheorem*{pf}{Proof of theorem 2}
\newtheorem{Th}{Theorem}[chapter]
\newtheorem*{ThR}{Theorem of Roth [9]}

\newtheorem{Cor}{Corollary}[chapter]

\newtheorem{Lem}{Lemma}
\newtheorem{Exp}{Example}
\newtheorem{Rem}{Remark}
\newtheorem{Rems}{Remarks}
\newtheorem{p}{Proof of theorem 1 }

\renewcommand\thesection{\ifnum \value{chapter}>0 \thechapter.\fi\arabic{section}}

\renewcommand\theDef{\ifnum \value{chapter}>0 \thechapter.\fi\arabic{Def}}

\renewcommand\theRem{\ifnum \value{chapter}>0 \thechapter.\fi\arabic{Rem}}

\renewcommand\theCor{\ifnum \value{chapter}>0 \thechapter.\fi\arabic{Cor}}

\renewcommand\theTh{\ifnum \value{chapter}>0 \thechapter.\fi\arabic{Th}}

\renewcommand{\theRem}{\empty{}}

\newcommand{\poubelle}[1]{}
\begin{document}
\renewcommand{\proofname}{\textbf{Proof}}

 \title{Transcendence theorem}
 \begin{center}
 \textbf{\Large{Transcendence of some infinite series}}
 \end{center}

 \begin{center}
     \textbf{\large{Fedoua Sghiouer, Kacem Belhroukia, Ali Kacha }}
 \end{center}

 \begin{center}
     Department of Mathematics, Ibn Tofail University,\\
     14 000 Kenitra, Morocco \\
     E-mail: fedoua.sghiouer@uit.ac.ma\\
     E-mail: belhroukia.pc@gmail.com\\
     E-mail: ali.kacha@uit.ac.ma\\

 \end{center}

 \begin{center}
     \textbf{\large{Abstract}}
 \end{center}

In the present paper and as an application of Roth's theorem concerning the rational approximation of algebraic numbers, we give a sufficient condition that will assure us that a series of positive rational terms is a transcendental number. With the same conditions, we establish a transcendental measure of $ \sum_{n = 1}^{\infty} 1/a_n$. \\

\hspace*{-0.6cm}\textbf{Keywords:} Infinite series, transcendental number, transcendental measure.

\section{Introduction and preliminaries }

The theory of transcendental numbers has a long history. We know since J. Liouville in 1844 that the very rapidly converging sequences of rational numbers provide examples of transcendental numbers. So, in his famous paper [6], Liouville showed that a real number admitting very good rational approximation can not be algebraic, then he explicitly constructed the first examples of transcendental numbers.\\

There are a number of sufficient conditions known within the literature for an infinite series, $ \sum_{n = 1}^{\infty} 1/a_n$, of positive rational numbers to converge to an irrational number, see [2, 8, 10].These conditions, which are quite varied in form, share one common feature, namely, they all require rapid growth of the sequence ${ a_n }$ to deduce irrationality of the series. As an illustration consider the following results of J. Sandor which have been taken from [10] and [11].\\

From this direction, the transcendence of some infinite series has been studied by several authors such as M.A. Nyblom [7], J. Hančl and J. Štěpnička [4]. we also note that the transcendence of some power series with rational coefficients is given by some authors such as J. P. Allouche [1] and G. K. Gözer [3].

\newpage
\begin{ThR}
\normalfont
 Let $ \alpha $ be a real number, $ \delta $ a real number > 2, if there exists an infinity rational numbers $ \dfrac{p}{q} $ with $ gcd(p, q) = 1 $ such that $$ \left| \alpha - \frac{p}{q} \right| < \frac{1}{q^{ \delta} }, $$

then $ \alpha $ is a transcendental number.
\end{ThR}

\poubelle{
\begin{Th}\textbf{( Eisenstein, 1852), ( Heine, 1853 )  }~\\

Any algebraic power series $ f = \sum_{n\geq 0} a_nt^n $ in $ \mathbb{Q}\left[\left[t\right]\right]$ is globally bounded :  There exists an integer $ C > 0 $ such that $ a_nC^n $ is an integer for all $ n \geq 1 $.

\end{Th}

Consequently, the power series:

\begin{enumerate}[label= \textbf{\roman*)}]
    \item $ \sum_n \frac{t^n}{n} $
    \item $ \sum_n \frac{t^n}{n^2 + 1 }$
    \item $ \sum_n \frac{t^n}{n!} $
\end{enumerate}
 are transcendental.
\begin{Rems}~\\
\begin{enumerate}[label= \textbf{\bullet}]
    \item The smallest possible constant C is called Eisenstein constant of $ f $.
    \item Best current bound is, see $\left[ Dwork,\hspace{0.1cm} van \hspace{0.1cm} der \hspace{0.1cm} Poorten \hspace{0.1cm} 1992 \right]$ $$ C \leq 4.8 \left(8e^{-3}D^{4 + 2.74 \log D} e^{1.22D} \right)^D. H^{2D - 1} = e^{O\left(D^2\right)}. H^{2D - 1} $$
    where $ D $ is the degree of the (minpoly of) $ f ,$ and $ H $ the max of its coefficients.
\end{enumerate}

\end{Rems}
}
\poubelle{
\begin{Rem}
The infinite series :

\begin{enumerate}[label= \textbf{\roman*)}]
\item $ \sum_{n \geq 0} \frac{1}{n!} = e,$
\item  $ \sum_{n \geq 0} \frac{1}{n^2 + 1} = \frac{\pi^2}{2} $

\end{enumerate}
are transcendental numbers.

\end{Rem}
}

\section{Main results}

\subsection{Transcendence}
The first main result is given in the following theorem.

\begin{Th}{\it  Let $(a_n)_{n \geq 1}$ be a sequence of non-zero natural integers and $ \alpha $ be a positive real\\ > 2 such that
\begin{equation}
a_{n + 1} > a_n^{\alpha + 1}, \hspace{0.1cm}\text{ for all } n \geq 1. \\
\end{equation}
 Then the series $ \sum_{n = 1}^ \infty 1/{a_n} $ converges to a transcendental number.}
\end{Th}

 In order to prove Theorem 1, we need some preliminary lemmas.
\begin{Lem}{\it
Let $$ \frac{p_m}{q_m} = \sum_{k = 1}^m \frac{1}{a_k} $$ such that $ (p_m, q_m) = 1$. Then, we have
\begin{equation}
   q_m \leq a_1a_2...a_m.
\end{equation}}
\end{Lem}

\begin{proof}
Since $ (p_m, q_m) = 1$,  the lowest common denominator of the fraction $\dfrac{1}{a_1}...\dfrac{1}{a_m}$ must be greater than or equal to $ q_m $. So we deduce $ q_m \leq a_1a_2...a_m$.

\end{proof}

\begin{Lem}{\it
Let $ (a_n)_{n \geq 1}$ be a sequence of natural integers > 0, and $\alpha$ be a given real > 2. The hypothesis $ a_{n+1} > a_n^{\alpha + 1}$ implies that
\begin{equation}
    \lim\limits_{n \rightarrow \infty} \frac{\left(a_1a_2...a_n \right)^\alpha}{a_{n+1}} = 0
\end{equation}}

\end{Lem}

\begin{proof}
We put $$ b_n = \frac{\left(a_1a_2...a_n \right)^{\alpha}}{a_{n+1}}, $$ and we show that $ \lim\limits_{n \rightarrow \infty} b_n = 0.$ We have
\begin{align*}
\ln \left(\frac{1}{b_n} \right) &= \ln \left(a_{n+1} \right) - \alpha \sum_{k = 1}^n \ln \left(a_k\right)\\
&= \sum_{k=1}^n \ln \left(\frac{a_{k+1}}{a_k} \right) + \ln\left(a_1\right) - \alpha \sum_{k=1}^n \ln\left(a_k\right)\\
&= \sum_{k=1}^n \ln\left(\frac{a_{k+1}}{a_k^{\alpha + 1}} \right) + \ln \left(a_1\right)\\
&\geq \sum_{k=1}^n \ln\left(\frac{a_{k+1}}{a_k^{\alpha + 1}}\right).
\end{align*}
Since $ \dfrac{a_{n+1}}{a_{n}^{\alpha +1}} > 1, $ then there exists $  \delta > 0 $ such that $ \dfrac{a_{n+1}}{a_{n}^{\alpha +1}} > 1 + \delta.$ Therefore, we get
\begin{equation*}
\ln\left(\frac{1}{b_n}\right) \geq n \ln\left( 1 + \delta \right).
\end{equation*}

From this, we deduce that, $ \lim\limits_{n \rightarrow +\infty} \ln\left(\dfrac{1}{b_n}\right) = +\infty  $, then $ \lim\limits_{n \rightarrow +\infty} b_n = 0$.\\
\end{proof}

\begin{p}\normalfont
According to the hypothesis, the series $ \sum_n \dfrac{1}{a_n}$ is convergent.\\
Put $ \theta = \sum_{n = 1} ^\infty \dfrac{1}{a_n} $ and $ \dfrac{p_m}{q_m} = \sum_{n = 1}^m \dfrac{1}{a_n} $. From the equality,

\begin{equation*}
\left| \theta - \frac{p_m}{q_m} \right| = \sum_{n = m +1}^\infty \frac{1}{a_n},
\end{equation*}
we obtain
\begin{equation*}
q_m^\alpha \left| \theta - \frac{p_m}{q_m} \right| = \sum_{n = m + 1}^\infty \frac{q_m^\alpha}{a_n}. \\
\end{equation*}
The relationship (2) implies that
\begin{equation*}
q_m^\alpha \left| \theta - \frac{p_m}{q_m} \right| \leq \left( a_1a_2...a_m \right)^\alpha
\end{equation*}
\begin{equation*}
\sum_{n = m + 1}^\infty \frac{1}{a_n} \leq b_m \sum_{n = m + 1}^\infty \frac{a_{m+1}}{a_n},
\end{equation*}
with $ b_m = \dfrac{\left(a_1a_2...a_m \right)^\alpha}{a_{m+1}}.$\\
Furthermore, we have
\begin{equation}
    \frac{a_n}{a_{n+1}} < \frac{1}{a_n^\alpha} < \frac{1}{a_n}, \hspace{0.08cm} all\hspace{0.08cm} n \geq 1.
\end{equation}
Then, we obtain
\begin{align*}
    q^\alpha_m \left| \theta - \frac{p_m}{q_m} \right| &< b_m \left( 1 + \sum_{k = 1}^\infty \frac{a_{m+1}}{a_{m+k+1}} \right)\\
    &< b_m \left( 1 + \sum_{k=1}^\infty \frac{a_{m+k}}{a_{m+k+1}} \right)\\
    &< b_m \left( 1 + \sum_{k=1}^\infty \frac{1}{a_{m+k}} \right) \\
    &< b_m \left( 1 + \theta \right). \hspace{2cm}
\end{align*}

According to the relationship (3), and for m sufficiently large, we get $ b_m < \left(1 + \theta \right)^{-1}.$\\

Therefore for m sufficiently large, we have $$ q_m^\alpha \left| \theta - \frac{p_m}{q_m}\right| < 1. $$

Finally we find
\begin{equation}
   \left| \theta - \frac{p_m}{q_m}\right| < \frac{1}{q_m^\alpha}.
\end{equation}

Since $ \alpha > 2 $, then by Roth's theorem, $ \theta $ is a transcendental number.\\

\end{p}

We will now give a corollary as an application of our main result.
\begin{Cor}{\it
Any subseries of the series $\sum_{n \geq 1} \frac{1}{a_n}$, where the terms $ a_n \in \mathbb{N}\setminus \{0\} $ satisfy (1) will have a transcendental sum.}
\end{Cor}

\begin{proof}
Consider an arbitrary subseries $ \sum_{n \geq 1 } 1/c_n $ then by definition there must exist a strictly monotone increasing function $ g : \mathbb{N} \rightarrow \mathbb{N}$ such that $ c_n = a_{g(n)} .$ Clearly as $ g\left( n + 1 \right) \geq g(n) + 1 $ one has $$ \frac{c_{n+1}}{c_n^{\alpha + 1}} = \frac{a_{g(n+1)}}{a_{g(n)}^{\alpha + 1 }} \geq \frac{a_{g(n) + 1 }}{a_{g(n)}^{\alpha + 1}} > 1, $$
 and  by Theorem 1 the subseries has a transcendental sum.
\end{proof}

\begin{Exp}\normalfont
We consider the following sequence: $$ \left\{
\begin{array}{rcr}
a_n = 2^{n! + 1}, & n \geq 1 \\
a_{n+1} > a_n^{3+\epsilon}, & n \geq 3
\end{array}
\right. $$

By applying Theorem 1, the series $$ \sum_n \frac{1}{2^{n! + 1}} $$  converge to a transcendental number.

\end{Exp}

\subsection{Transcendental measure}
The second main result of this paper is to give a transcendental measure of $ \theta = \sum_{n=1}^{\infty} \frac{1}{a_n} $. In this subsection, we keep the same notations as in the first subsection.\\

\begin{Th}
Let $P \in \mathbb{Z}[X]\textbackslash \left\{  0 \right \} $ be a polynomial of degree $ d \geq 2$ and height $ H $. Let $\alpha > d $ and $k > 1$ be two real numbers such that  $$ a_n^{\alpha + 1 } \leq a_{ n + 1 } < a_n^{k \alpha }, \hspace{0.08cm}\text{ for all } n \geq 1. $$
Then, we have $$ \lvert P(\theta ) \rvert > \dfrac{1}{\left(Hd\left(d+1\right)\right)^{\frac{kd\left(\alpha + 1\right)}{\alpha - d}}}. $$
\end{Th}

 To prove this Theorem, we need the following lemma.\\

\begin{Lem}
\begin{enumerate}[label = \textbf{(\roman*)}] \item The hypothesis  $\hspace{0.2cm} a_n^{\alpha + 1 } \leq  a_{ n + 1 }$ implies that
    \begin{equation}
        q_n \leq a_n^{\frac{\alpha + 1 }{\alpha}}, for\hspace{0,1cm}  n \geq 1.
    \end{equation}
    \item The hypothesis $\hspace{0.2cm} a_{ n + 1 } <a_n^{k \alpha } $ implies that
    \begin{equation}
        q_{n+1} < q_n^{k(\alpha + 1 )}, for  \hspace{0,1cm} n \geq 1.
    \end{equation}
\end{enumerate}
\end{Lem}

\begin{proof}
\begin{enumerate}[label = \textbf{\roman*)}]\item The hypothesis of i) implies that  $$ \hspace{-1.7cm}a_n \leq a_{n + 1 }^{\frac{1}{\alpha + 1}} . $$ Then for all $ 1 \leq j \leq n-1,$ we obtain $$ a_j \leq a_n^{\left(\frac{1}{\alpha +1 }\right)^
{n-j}}. $$
On the other hand, according to the relationship (2), one has $$\hspace{-3cm} q_n \leq a_1...a_{n-1}.a_n, $$ this implies
    $$ \hspace{-2cm} q_n \leq a_n^{1 + \frac{1}{\alpha + 1} + \frac{1}{\left( \alpha + 1 \right)^2} + ... + \frac{1}{\left( \alpha + 1 \right)^{n - 1}}}.$$ Which gives
    $$  \hspace{-2cm} q_n \leq a_n^{\frac{1}{1 - \frac{1}{\alpha + 1}}}. $$
    Finally we obtain $$ \hspace{0.8cm} q_n \leq a_n^{\frac{\alpha + 1 }{\alpha}}, \hspace{0.5cm} for \hspace{0.1cm} all \hspace{0.1cm} n \geq 1. $$
\item According to the relationship (6), we have
\begin{align*}
    q_n &\leq a_n^{\frac{\alpha + 1 }{\alpha}} < a_{n-1}^{\frac{\alpha +1}{\alpha}k\alpha} = a_{n-1}^{k\left( \alpha + 1 \right)}.
\end{align*}
Since $a_n < q_n$ for all $ n \geq 1$, we obtain
$$ q_n < q_{n - 1}^{k\left(\alpha + 1 \right)}. $$
\end{enumerate}

\end{proof}

\begin{pf}\normalfont
Put $$  \theta_n = \frac{p_n}{q_n} = \sum_{k=1}^{n} \frac{1}{a_n}.  $$
From the equality,  $$ P\left(\theta_n\right) =  P\left(\theta_n\right) -  P\left(\theta\right) +  P\left(\theta\right), $$
we get
\begin{equation*}
    \left| P(\theta_n ) \right| \leq \left| P(\theta_n ) - P(\theta ) \right| +\left| P(\theta )\right|.
\end{equation*}
Therefore,

\begin{equation}
    \left| P(\theta ) \right| \geq \left| P(\theta_n ) \right| - \left| P(\theta ) - P(\theta_n ) \right|.
\end{equation}
Put $ \hspace{0.2cm}P = \sum_{k=1}^{d} e_kX^k $, then
\begin{align}
    P(\theta_n ) &= P\left( \frac{p_n}{q_n}\right) = \sum_{k=1}^{d} e_k \frac{p_n^k}{q_n^k} = \frac{1}{q_n^d}\sum_{k=1}^{d} e_kp_n^kq_n^{d-k}.
\end{align}
Notice that
\begin{equation*}
    \sum_{k=1}^{d} e_kp_n^kq_n^{d-k} \neq 0,
\end{equation*}
because, if we assume that $$ \sum_{k=1}^{d} e_kp_n^kq_n^{d-k} = 0, $$\\
then we would have
\begin{align*}
    q_n^d.P\left(\theta_n\right) &= 0, \\
\end{align*}
which implies that
\begin{align*}
P\left(\theta_n\right) &= 0.
\end{align*}
 We also have $$ P\left(\theta\right) = P\left( \lim\limits_{n \rightarrow +\infty} \theta_n\right) = \lim\limits_{n \rightarrow +\infty} P\left(\theta_n\right) = 0, $$\\
which contradicts the fact that $ \theta $ is a transcendental number. Therefore, $$ \left| \sum_{k=1}^{d} e_kp_n^kq_n^{d-k} \right| \geq 1. $$\\
According to equality (9), we get

\begin{equation}
    \left| P\left(\theta_n\right)\right| \geq \frac{1}{q_n^d}.
\end{equation}
On the other hand, according to the mean value theorem applied to $P$, there exists a real number $ F \in ]\theta_n, \theta[ $ or  $]\theta, \theta_n[  $ such that
\begin{equation*}
    P\left( \theta \right) - P\left( \theta_n \right) = P'\left(F\right) \left( \theta - \theta_n \right).
\end{equation*}
From this, we obtain\\
\begin{equation}
    \left| P\left( \theta \right) - P\left( \theta_n \right) \right| = \left| P'\left(F\right) \right| \left| \theta - \theta_n \right|.
\end{equation}
Furthermore, as
\begin{align*}
    P'\left(F\right) &= \sum_{k=1}^{d} k F^{k-1}e_k,
\end{align*}
which implies that
\begin{align*}
     \left| P'\left(F\right) \right| &\leq \sum_{k=1}^{d} k \left| F\right|^{k-1} \left| e_k \right|\\
    &\leq \sum_{k=1}^{d} k \left| e_k \right| \leq H \sum_{k=1}^{d} k\\
    &\leq H\dfrac{d(d+1)}{2}.
\end{align*}
Therefore, the equality (11) becomes
\begin{equation}
    \left| P\left( \theta \right) - P\left( \theta_n \right) \right| < H\dfrac{d(d+1)}{2} \left| \theta - \theta_n \right|.
\end{equation}
According to the relationship (5), we have

\begin{equation*}
    \left| \theta - \theta_n \right| < \frac{1}{q_n^\alpha},
\end{equation*}
then
\begin{equation}
    \left| P\left( \theta \right) - P\left( \theta_n \right) \right| < \dfrac{Hd(d+1)}{2q_n^\alpha}.
\end{equation}
By combining (10) and (13), the relationship (8) becomes
\begin{equation*}
     \left| P\left( \theta \right) \right| > \frac{1}{q_n^d} - \dfrac{Hd(d+1)}{2q_n^\alpha},\hspace{0.06cm} \text{for n sufficiently large.}
\end{equation*}
In order to have $\hspace{0.05cm} \left| P\left( \theta \right) \right| > \dfrac{1}{2q_n^d},$ it suffices to have
\begin{align*}
    \frac{1}{q_n^d} - \dfrac{Hd(d+1)}{2q_n^\alpha} > \frac{1}{2q_n^d},\\
\end{align*}
which is equivalent to
\begin{align*}
    \frac{1}{2q_n^d} &> \dfrac{Hd(d+1)}{2q_n^\alpha}
    \Longleftrightarrow  q_n^{\alpha - d } > Hd(d+1).
\end{align*}
So that, we take $n_1$ the smallest integer such that

\begin{equation}
    q_{n_1-1}^{\alpha - d} < Hd(d+1) < q_{n_1}^{\alpha - d}.
\end{equation}

\begin{Rem}\normalfont
\textrm{The natural number $n_1$ exists because $\lim\limits_{n \rightarrow \infty} q_n^{\alpha - d} = + \infty$, then, we obtain $$q_{n_1}^{\alpha - d} > Hd(d+1).$$}
\end{Rem}
Therefore, we get
\begin{equation}
    \left| p\left( \theta \right) \right| > \frac{1}{2q_{n_1}^d}.
\end{equation}
Using the (ii) from Lemma 3, the relationship (15) becomes
\begin{equation}
    \left| P\left(\theta\right ) \right| > \frac{1}{2q_{n_1}^d} > \frac{1}{2q_{n_1 - 1}^{kd\left(\alpha + 1 \right)}} .
\end{equation}
According to the relationship (14), we have
\begin{align*}
    \frac{1}{q_{n_1-1}^{\alpha - d}} &> \frac{1}{Hd(d+1)},
\end{align*}
then
\begin{align*}
    \frac{1}{q_{n_1-1}^{kd\left(\alpha + 1\right)}} &>\frac{1}{\left(Hd\left(d+1\right)\right)^{\frac{kd\left(\alpha + 1\right)}{\alpha - d}}}.
\end{align*}
So, the relationship (16) becomes
\begin{align*}
     \left| P\left(\theta\right ) \right| &> \dfrac{1}{\left(Hd\left(d+1\right)\right)^{\frac{kd\left(\alpha + 1\right)}{\alpha - d}}},
\end{align*}
which completes the proof of Theorem 2.
$\hspace*{15cm} \Box$
\end{pf}

\begin{Exp}\normalfont
Let $$ \left\{
    \begin{array}{ll}
     a_0 = 0, a_1 = 2,\\
        a_{n+1} = a_n^4, n \geq 1,  \\
        \alpha = 4, k = 2.
    \end{array}
\right. $$

Let $ P \in \mathbb{Z}[X]\textbackslash\left\{ 0 \right\}$ be a quadratic polynomial of height $ H $. By applying Theorem 2, a transcendental measure of $ \theta = \sum_{n = 1}^{\infty} 1/a_n$ is given by

    $$\LARGE{\left| P(\theta) \right| > \dfrac{1}{(6H)^{10}} .}$$

\end{Exp}

\end{document}